\documentclass[pdflatex,11pt,twoside]{article}
\usepackage[dvips]{color,graphics}
\usepackage{multicol}
\usepackage{epsfig}
\usepackage{amssymb}
\usepackage{amsmath}
\usepackage{amsfonts,verbatim}
\usepackage{sectsty}
\usepackage{graphicx}
\input labelfig.tex
\sectionfont{\rm}
\subsectionfont{\rm}
\subsubsectionfont{\rm}

\newcommand{\R}{\mathbb{R}}
\newfont{\footsc}{cmcsc10 at 8truept}
\newfont{\footbf}{cmbx10 at 8truept}
\newfont{\bigrm}{cmr12 scaled\magstep4}
\newfont{\footrm}{cmr10 at 10truept}
\definecolor{DarkGreen}{rgb}{0,0.7,0}

\newcounter{foo-theorem}
\newcounter{foo-corollary}
\newcounter{foo-proposition}
\newcounter{foo-lemma}
\newcounter{foo-conjecture}
\newfont{\blb}{msbm10 at 11truept}
\newfont{\comp}{cmr12 scaled\magstep1}
\newfont{\compb}{cmr10 scaled\magstep2}
\newtheorem{theorem}[foo-theorem]{Theorem}
\newtheorem{lemma}[foo-lemma]{Lemma}

\newtheorem{proposition}[foo-proposition]{Proposition}
\newenvironment{definition}{{\bf Definition.}}{}
\newenvironment{proof}{{\bf Proof.}}{\hfill{ }\vrule height10pt width5pt depth1pt \medskip}

\date{}

\begin{document}

\author{Benny Sudakov \thanks{
Department of Mathematics, UCLA, Los Angeles, CA 90095. Email: bsudakov@math.ucla.edu. Research
supported in part by NSF CAREER award DMS-0812005 and by a USA-Israeli BSF grant.}
\and Jacques Verstra\"{e}te
\thanks{Department of Mathematics, University of
California, La Jolla, CA, 92093.
E-mail: jverstraete@ucsd.edu. Research supported in part by NSF Grant DMS-0800704 and an Alfred P. Sloan Research Fellowship.}}

\title{Cycles in sparse graphs II}

\maketitle

\vspace{-0.3in}

\begin{abstract}
The {\em independence ratio} of a
graph $G$ is defined by
\[ \iota(G) := \sup_{X \subset V(G)} \frac{|X|}{\alpha(X)},\]
where $\alpha(X)$ is the independence number of the subgraph of $G$ induced by $X$. The independence ratio is a relaxation of the chromatic number $\chi(G)$
in the sense that $\chi(G) \geq \iota(G)$ for every graph $G$, while for many natural classes of graphs
these quantities are almost equal. In this paper, we address two old conjectures of Erd\H{o}s on cycles in graphs with large chromatic number
and a conjecture of Erd\H{o}s and Hajnal on graphs with infinite chromatic number.
\end{abstract}

\bigskip

\section{Introduction}

Let $G$ be a graph and let $\alpha(X)$ be the size of a
largest independent set in the subgraph of $G$ induced by $X$.
The {\em independence ratio} of a graph $G$ is defined by
\[ \iota(G) := \sup_{X\subset V(G)} \frac{|X|}{\alpha(X)}.\]
The independence ratio of a graph $G$ is a relaxation of the chromatic number $\chi(G)$, since
$\chi(G) \geq \iota(G)$ for all graphs $G$. For many interesting classes of graphs, including
random and pseudorandom graphs, the chromatic number and independence ratio
are equal or almost equal. On the other hand, so-called {\em Kneser graphs}
are examples of graphs on $n$ vertices with constant independent ratio and chromatic number of order $\log n$~\cite{L}.
In this paper, we are motivated by three conjectures on cycles in graphs with
large chromatic number. We give partial evidence for the truth of each conjecture by considering
graphs with large independence ratio.

\subsection{Erd\H{o}s' conjecture on many cycles}

Erd\H{o}s~\cite{Er1} conjectured that a triangle-free graph with chromatic number $k$ contains cycles of
at least $k^{2 - o(1)}$ different lengths as $k \rightarrow \infty$.
The conjecture of Erd\H{o}s remains open, and in fact no lower bound better than linear in $k$ is known for the longest cycle in a
triangle-free graph with chromatic number $k$. In general, Theorem 2 in~\cite{SV} shows that a graph of chromatic number $k$ and no cycle of length $g$ contains cycles of $\Omega(k^{\lceil (g - 1)/2 \rceil})$ different lengths. In this paper, we prove the following theorem which shows in a very strong sense that Erd\H{o}s' conjecture is true for graphs with large independence ratio:

\begin{theorem}\label{long}
Every triangle-free graph with independence ratio at least $k \geq 3$
has cycles of $\Omega(k^2 \log k)$ consecutive lengths.
\end{theorem}

We prove this theorem in Section 2.5, after some preliminary results in Sections 2.1 -- 2.4.
The important result by Kim~\cite{Kim} establishing the order of magnitude of triangle-complete graph Ramsey numbers $r(3,t) = \Theta(t^2/\log t)$ shows that there are triangle-free graphs with independence ratio $k$ and with $O(k^2 \log k)$ vertices, so the above
result is best possible up to the value of the implicit constant.

\subsection{Hereditary Properties}

Theorem \ref{long} is part of a more general theorem on {\em hereditary properties} -- families of graphs closed under taking induced subgraphs. To describe the general theorem, let $P$ be a hereditary
property and let $f : [1,\infty) \rightarrow [1,\infty)$ be an increasing bijection. Then we say that $P$ has {\em speed at most $f$} if for every $n \in \mathbb N$ and every $n$-vertex graph
$G \in P$, we have $\iota(G) \leq f(n)$. Since the identity function $f(x) = x$ for $x \in [1,\infty)$ serves as an upper bound for the speed of every hereditary property, the speed of each hereditary
property is well-defined. We shall prove the following theorem:

\begin{theorem}\label{general}
Let $f : [1,\infty) \rightarrow [1,\infty)$ be an increasing bijection. If $P$ is a hereditary property with speed at most $f$, then any graph $G \in P$ with $\iota(G) > 18k + 4$ has cycles of at least
$\frac{1}{2}f^{-1}(k)$ consecutive lengths.
\end{theorem}

This theorem is proved in Section 2.6. Theorem \ref{general} applies in general to the
property $P_H$ of $H$-free graphs -- this is the hereditary property of graphs which do not contain any copy of $H$.
For example, one can obtain an appropriate generalization of Theorem \ref{long}.

\begin{theorem}\label{kskt}
Let $G$ be a $K_{s+1}$-free graph and suppose $\iota(G) > 18k + 4$. Then $G$ contains cycles of
at least $\frac{1}{2}(k/s)^{s/(s - 1)}$ consecutive lengths.
\end{theorem}

Clearly this theorem holds also for all $H$-free graphs where $H$ has $s+1$ vertices.
We prove Theorem \ref{kskt} in Section 2.7. One can improve the lower bound $(k/s)^{s/(s - 1)}$ in the above theorem by
a polylogarithmic factor which, for $s = 3$, agrees with Theorem \ref{long}, but
this involves only further computations and so will be omitted.

\subsection{Erd\H{o}s' conjecture on unavoidable cycles}

Erd\H{o}s~\cite{E} offered one thousand dollars for a satisfactory resolution of the
following problem: in a graph of infinite chromatic number, which
cycle lengths should appear? For example, Erd\H{o}s conjectured that a graph of sufficiently
large chromatic number has a cycle of length a prime.
We show that if an $n$-vertex graph has independence ratio at least $3\exp(8\log^*\!n)$, then not only does it contain a cycle of prime length, it contains cycles of lengths from many other sparse infinite sequences of integers. Here $\log^*\!n$ is the number of times the natural logarithm must be applied to $n$ to get a number less than one, and in what follows, we write $\log_b n$ for the logarithm base $b$, and omit the base if the logarithm is the natural logarithm. We prove the following theorem:

\medskip

\begin{theorem}\label{lengths}
Let $\sigma$ be an infinite increasing sequence of positive integers satisfying $\sigma_1 \geq 3$ and $\log \sigma_r \leq \sigma_{r - 1}$ for all $r \geq 2$. If $G$ is an $n$-vertex graph and
\[ \iota(G) \geq \sigma_1 \exp(8\log^*\!n),\]
then $G$ contains a cycle of length in the sequence $\sigma$.
\end{theorem}

\medskip

Theorem \ref{lengths} is proved in Section 3. We claimed that an $n$-vertex graph $G$ with $\iota(G) \geq 3\exp(8\log^*\!n)$ contains a cycle of length a prime. Let $p_r$ denote the $r$th prime number. Then Bertrand's Postulate gives $p_{r + 1} \leq 2p_r$ for all $r \in \mathbb N$, and so $\log p_{r + 1} \leq \log p_r + 1 \leq p_r$ for all $r \in \mathbb N$. Applying Theorem \ref{lengths} to this sequence, we see that a graph $G$ with $\iota(G) > 3\exp(8\log^*\!n)$ contains a cycle of length a prime, as claimed.
Theorem \ref{lengths} gives a similar upper bound for much sparser sequences, such as powers of three, or $2 + 1,2^2 + 1,2^{2^2} + 1,\dots$ and so on.

\bigskip

An important remark is that Theorem \ref{lengths} distinguishes between the independence ratio and the chromatic number: generalizations of Mycielski's well-known construction of triangle-free graphs of
arbitrarily large chromatic number provide constructions for infinitely many $n$ of an $n$-vertex graph $G_n$ of chromatic number $\chi(G_n) = \Omega((\log n)/(\log\log n))$ with no cycle of length in a
prescribed sequence $\sigma$ satisfying $\log \sigma_r \leq \sigma_{r - 1}$ for $r \geq 2$. The conclusion of Theorem \ref{lengths} therefore does not hold if we replace $\iota(G)$ with $\chi(G)$ in the
theorem. We present the details of this construction in Section 5.

\subsection{Erd\H{o}s-Hajnal conjecture}

Let $C(G) = \{\ell : C_{\ell} \subset G\}$ denote the set of cycle lengths in a graph $G$.
Erd\H{o}s~\cite{E} proposed the study of the quantity
\[ L(G) = \sum_{t \in C(G)} \frac{1}{t}\]
and conjectured that in a graph of infinite chromatic number,
$L(G)$ diverges. This conjecture was proved by Gy\'{a}rf\'{a}s, Komlos and Szemer\'{e}di~\cite{GKS}.
Specifically, they proved that if $G$ is a finite graph of minimum degree $d$,
then there exists $\epsilon > 0$ such that
\[ L(G) \geq \epsilon \log d. \]
This is best possible up to the value of the constant $\epsilon$, since the complete bipartite graph has $L(K_{d,d}) \leq \frac{1}{2}\log d + 1$.
It follows that $L(G)$ diverges when $G$ has infinite chromatic number, since a graph of infinite chromatic number contains a graph of minimum degree at least $d$ for each $d \in \mathbb N$.
The result above therefore does not rely on the chromatic number as much as the existence of subgraphs of arbitrarily large average degree.
Erd\H{o}s and Hajnal~\cite{E,EH} conjectured that in a graph with infinite chromatic number, the sum of reciprocals of odd cycle lengths diverges.
If $C_{\circ}(G)$ is the set of lengths of odd cycles in $G$, their conjecture states that if $G$ is a graph of infinite chromatic number, then
\[ L_{\circ}(G) := \sum_{t \in C_{\circ}(G)} \frac{1}{t} = \infty.\]
In this paper, we give some evidence for this conjecture by showing:

\medskip

\begin{theorem}\label{recip}
For any graph $G$ on $n$ vertices,
\[ L_{\circ}(G) \; \geq \;
\frac{1}{2} \log \iota(G) - 8\log^*\!n.\]
\end{theorem}

\medskip

We prove Theorem \ref{recip} in Section 4.
This theorem is best possible up to the $O(\log^*\!n)$ term, in the sense that $L_{\circ}(K_t) \leq \frac{1}{2} \log \iota(K_t) + 1$.
It would be interesting as a first step to the Erd\H{o}s-Hajnal conjecture to show that $L_{\circ}(G)$ diverges when $G$
is a graph with infinite independence ratio.

\section{Preliminary results}

In this section, we present the results necessary for the proofs of Theorems \ref{long}--\ref{recip}.
The following notation will be used. If $G$ is a graph, then $\iota(G)$ is its independence ratio and $\alpha(G)$ is its
independence number. For a set $X \subset V(G)$, we denote by $\partial_G X$ the set of vertices
of $V(G) \backslash X$ adjacent to at least one vertex in $X$. We sometimes omit the subscript $G$ when it is clear which graph is
 being referred to. We write $\alpha(X)$ for the independence number of the subgraph of
$G$ induced by $X$. Fixing a vertex $v \in V(G)$, it is convenient to let $N_i(v)$ denote the set of vertices at
distance exactly $i$ from $v$.

\subsection{Expanding subgraphs}

The starting point for proving Theorems \ref{long} -- \ref{lengths} is to show that graphs with large independence ratio
have nice expansion properties. Precisely, we make the following definitions:

\bigskip

\begin{definition}
We say that a graph $G$ is {\em $k$-expanding on
independent sets} if every independent set $I$ in $G$ has $|\partial_G I| > k|I|$. A graph $G$ is {\em weakly $k$-expanding
on independent sets} if for some $v \in V(G)$, every independent set $I \subset V(G) \backslash \{v\}$ has $|\partial_G I| > k|I|$.
\end{definition}

\bigskip

We notice in particular that if a graph is weakly $k$-expanding on independent sets, then all but at most one vertex in the graph has degree more than $k$.

\begin{lemma}\label{indep-expand}
Let $k \geq 1$. Then every $n$-vertex graph $G$ with $\alpha(G) < n/(k + 1)$ has an induced subgraph that is $k$-expanding on independent sets and
a 2-connected subgraph that is weakly $k$-expanding on independent sets.
\end{lemma}

\begin{proof}
First we show that $G$ has a subgraph $H$ that is $k$-expanding on independent sets. Let $G_0 = G$.
If $G_0$ has no subgraph $H$ as above, then there is an independent set $I_0 \subset V(G_0)$ such that $|\partial I_0| \leq k|I_0|$. Let $G_1 = G - I_0 - \partial I_0$. Then there is an independent set $I_1 \subset V(G_1)$
with $|\partial I_1| \leq k|I_1|$. Let $G_2 = G_1 - I_1 - \partial I_1$.
Continuing in this way, we eventually remove independent sets
$I_0,I_1,\dots,I_r$ and their neighborhoods $\partial I_0,\partial I_1,\dots,\partial I_r$ and this exhausts all the vertices in the graph:
\[ V(G_0) = \bigcup_{j=0}^{r} (I_j \cup \partial I_j).\]
However, the set $I = \bigcup_{j=0}^{r} I_j$ is an independent set of size at least
$n/(k + 1)$ in $G = G_0$, which is a contradiction. Therefore there exists $H \subset G$ that is $k$-expanding on
independent sets. This proves the first statement of the lemma.

\medskip

To prove the second statement, if $H$ is 2-connected, then we are done. If $H$ is not 2-connected, let $F$ be an {\em endblock} of $H$ --
this is a maximal 2-connected subgraph of $H$ containing exactly one cut vertex $v$ of $H$. Then for any independent set $I \subset V(F) \backslash \{v\}$,
\[ |\partial_F I| = |\partial_H I| > k|I|.\]
So $F$ is weakly $k$-expanding on independent sets. This completes the proof.
\end{proof}

\bigskip

Ajtai, Koml\'{o}s and Szemer\'{e}di~\cite{AKS} showed that if $G$ is an $n$-vertex triangle-free graph, then $\alpha(G) = \Omega(\sqrt{n\log n})$. Their result was improved by Shearer~\cite{Sh}, who showed that if $G$ is an $n$-vertex triangle-free graph of maximum degree $d \geq 2$, then
\[ \alpha(G) \geq \frac{n(d\log d - d + 1)}{(d - 1)^2}.\]
A straightforward calculation gives the following result:

\begin{lemma}\label{ramsey-triangles}
For $n \geq e^{15}$, every $n$-vertex triangle-free graph $G$ has
\[ \alpha(G) > \Bigl(\frac{n \log n}{2}\Bigr)^{1/2}.\]
\end{lemma}

\begin{proof}
Let $m = (\frac{n \log n}{2})^{1/2}$.
If $G$ is a triangle-free $n$-vertex graph of maximum degree $d$,
then $\alpha(G) \geq n/(d + 1)$ and the lemma follows easily for $d \leq 1$.
Suppose $d \geq 2$. Observing that the neighborhood of any
vertex of $G$ is an independent set, we obtain from Shearer's bound,
\[ \alpha(G) \geq \max\Bigl\{d,\frac{n(d\log d - d + 1)}{(d - 1)^2}\Bigr\}.\]
If $d > m$, then we easily recover the bound in the lemma . If $2 \leq d \leq m$,
the second expression is a decreasing function of $d$ and so it is
minimized when $d = m$. In this case we get
\begin{eqnarray*}
\alpha(G) &\geq& \frac{n(m\log m - m + 1)}{(m - 1)^2} \\
&>& \frac{n (\log m - 1)}{m} \\
&\geq& \Bigl(\frac{n \log n}{2}\Bigr)^{1/2} + \frac{(\log \log n - 2 - \log 2) n^{1/2}}{(2\log n)^{1/2}}.
\end{eqnarray*}
Since $n \geq e^{15}$, $\log\log n \geq 2 + \log 2$, and so we have the required bound.
\end{proof}

The second definition we require is that of expansion on arbitrary sets of vertices in a graph:

\bigskip

\begin{definition}
A graph $G$ is {\em $k$-expanding on sets of size at most $T$} if
$|\partial_G X| > k|X|$ for every $X \subset V(G)$ of size at most $T$.
A graph $G$ is {\em weakly $k$-expanding on sets of size at most $T$} if for some $v \in V(G)$, every set $X \neq \{v\}$ of size at most $T$ has $|\partial_G X| > k|X|$.
\end{definition}

\medskip

\begin{lemma}\label{monotone-triangle}
If $G$ is an $n$-vertex triangle-free graph with $\alpha(G) < n/(3k + 1)$ and $k \geq e^{15}$, then $G$ contains a 2-connected subgraph $H$ that is weakly $2$-expanding on sets of size at most $k^2 \log k$ and weakly $3k$-expanding on independent sets.
\end{lemma}

\begin{proof}
Pass to a 2-connected subgraph $H \subset G$ which is weakly $3k$-expanding on independent sets, using Lemma \ref{indep-expand}.
Then for some $v \in V(H)$, every independent set $I \subset V(H) \backslash \{v\}$ has $|\partial_H I| > 3k|I|$.
Let $X \subset V(H)$ satisfy $|\partial_H X| \leq 2|X|$ where $X \neq \{v\}$. If $u \in X \backslash \{v\}$, then
\[ |\partial_H X| \geq |\partial_H \{u\}| - |X| > 3k - |X|.\]
If $|X| < e^{15}$, then since $k \geq e^{15}$ we obtain $|\partial_H X| > 2|X|$. If $|X| \geq e^{15}$, then
$H[X]$ is a triangle-free graph with at least $e^{15}$ vertices. Suppose $|X \backslash \{v\}| = x$. By Lemma \ref{ramsey-triangles}, there is an independent set $I \subset X \backslash \{v\}$ such that
\[ |I| \geq \Bigl(\frac{x \log x}{2}\Bigr)^{1/2} \geq 3,\]
Consequently, $|X| - |I| \leq x - 2$ and therefore
\begin{eqnarray*}
2|X|=2x + 2 &\geq& |\partial_H X| \\
&\geq& |\partial_H I| - |X| + |I| \\
&>& 3k|I| - x + 2 \\
&\geq& 3k \cdot \Bigl(\frac{x \log x}{2}\Bigr)^{1/2} - x + 2.
\end{eqnarray*}
This implies $2x > k^2 \log x$ and therefore $x > k^2 \log k$. This completes the proof.
\end{proof}

\subsection{P\'{o}sa's lemma and long cycles}

A well-known result of P\'{o}sa~\cite{P} shows how to find long paths in graphs which are 2-expanding on sets of size at most $T$: in this case one obtains a path of length at least $3T$. To describe P\'{o}sa's Lemma and the variant we use, we require some notation. If $P = v_1 v_2 \dots v_m$ is a longest path in a graph $G$ and $\{v_1,v_i\}$ is an edge of $G$, consider a path $Q = v_{i-1} \dots v_1 v_i \dots v_m$ of the same length as $P$, obtained by adding edge $\{v_1,v_i\}$ and
deleting edge $\{v_{i-1},v_i\}$ from $P$. We say that $Q$ was obtained from $P$ via an {\em elementary rotation}, which keeps
endpoint $v_m$ fixed. The set of all vertices of $P$ which are endpoints of paths obtained
by repeated elementary rotations from $P$ with fixed endpoint $v_m$ is denoted by $S(P)$. The following variant of
P\'{o}sa's Lemma (see Lemma 2.7 in~\cite{BD}) is our starting point:

\begin{proposition}\label{posa-original}
Let $T \geq 1$, and let $G$ be a graph that is 2-expanding on sets of size at most $T$. Then for any longest path $P \subset G$ there is a cycle $C \subset H$ of length at least $3T$ containing $S(P) \cup \partial S(P)$.
\end{proposition}

We require a slight adjustment of this proposition to accommodate weak 2-expansion: recall that $G$ is weakly 2-expanding on sets of size at most $T$
if for some $v \in V(G)$ and every set $X \subset V(G)$ with $X \neq \{v\}$, $|\partial_G X| > 2|X|$. The proof of the proposition below
is almost identical to the proof of Proposition \ref{posa-original} given in~\cite{BD}.

\begin{proposition}\label{posa}
Let $T \geq 1$, and let $G$ be a graph that is weakly 2-expanding on sets of size at most $T$. Then there exists a longest path $P \subset G$ and a cycle $C \subset H$ of length at least $3T$ containing $S(P) \cup \partial S(P)$.
\end{proposition}

\begin{proof}
By definition of weak expansion, for some $v \in V(G)$ and every $X \subset V(G)$ with $X \neq \{v\}$ we have $|\partial_G X| > 2|X|$.
Let $P$ be a longest path in $G$ and let $v_1 \neq v$ be an endpoint of $P$ and $v_m$ the other endpoint.
The crucial step in the proof of Proposition \ref{posa-original} is that $\partial S(P) \subset S(P)^- \cup S(P)^+$, where
$S(P)^-$ is the set of vertices preceding $S(P)$ on $P$ and $S(P)^+$ is the set of vertices succeeding $S(P)$ on $P$.
Since $|S(P)^-| \leq |S(P)|$ and $|S(P)^+| \leq |S(P)|$, we must have $|\partial S(P)| \leq 2|S(P)|$. Since $G$ is weakly 2-expanding
on sets of size at most $T$, we must have $S(P) = \{v\}$ or $|S(P)| > T$. However since $v_1 \neq v$, $|\partial_G \{v_1\}| > 2$ which shows that
$|S(P)| > 1$ and therefore $S(P) \neq \{v\}$. It follows that $|S(P)| > T$.
To construct a cycle containing $S(P) \cup \partial S(P)$,
let $y$ be the last vertex of $\partial S(P)$ on $P$. The segment of $P$ from $v_1$ to $y$ contains all vertices of $S(P) \cup \partial S(P)$,
otherwise any vertex of $S(P)$ after $y$ on $P$ would be distinct from $v_m$ and then the vertex after it on $P$ would be an element
of $\partial S(P)$, contradicting that $y$ is the last vertex of $\partial S(P)$ on $P$. If $x \in S(P)$ is a neighbor of $y$
in $G$, and $Q$ is a path from $x$ to $v_m$ obtained by elementary rotations from $P$, then $Q$ contains the segment of $P$ from $y$ to $v_m$.
Moreover, starting from $x$, $Q$ traverses all vertices of $S(P) \cup \partial S(P)$ before reaching $y$ and then continues along segment of $P$ from $y$ to $v_m$.
So $Q$ together with the edge $\{x,y\}$ forms the required cycle.
Since $|S(P) \cup \partial S(P)| \geq 3T$, this cycle has the required length.
\end{proof}

\medskip

Combining Lemma \ref{monotone-triangle} with Proposition \ref{posa}, we arrive at the following theorem, which
shows that triangle-free graph with large independence ratio contains a very long cycle.

\begin{theorem}\label{long2}
Let $G$ be a triangle-free graph with $\iota(G) > 3k + 1$ where $k \geq e^{15}$. Then $G$ has a cycle of length at least $3k^2 \log k$.
\end{theorem}

\begin{proof}
By Lemma \ref{monotone-triangle}, there is a 2-connected graph $H \subset G$ that is weakly $2$-expanding on sets of size at most $k^2 \log k$. By Proposition \ref{posa}, there is a cycle $C \subset H$ of length at least $3k^2 \log k$ in $G$, completing the proof.
\end{proof}

\subsection{Long odd cycles with chords}

A {\em chord of a cycle} is an edge which joins two
non-consecutive vertices of the cycle.
To prove Theorem \ref{long}, we need to extend Theorem \ref{long2} further to obtain a long odd cycle with a chord.

\begin{proposition}\label{theta}
Let $k \geq e^{15}$ and let $G$ be an $n$-vertex triangle-free graph with $\alpha(G) < n/(3k + 1)$. Then $G$ contains a non-bipartite graph consisting of a cycle of length at least $\frac{1}{2}k^2 \log k$ with at least one chord.
\end{proposition}

\begin{proof} By Lemma \ref{monotone-triangle}, $G$ contains a 2-connected subgraph $H$ that is weakly $2$-expanding on sets of size
at most $k^2 \log k$ and weakly $3k$-expanding on independent sets. By Proposition \ref{posa}, there is a cycle $C \subset H$ of
length at least $3T := 3k^2 \log k$ containing $S(P) \cup \partial_H S(P)$ for some longest path $P$ with $|S(P)| \geq T$. The
vertices of $S(P)$ are called {\em special vertices}, and since $\partial_H S(P) \subset V(C)$, all their neighbors are vertices of
$C$. Since $H$ is weakly $3k$-expanding on independent sets, all but at most one vertex of $H$ has degree more than $3k \geq e^{15}$.
In particular, every special vertex has two neighbors on $C$ which are not adjacent to that special vertex on $C$. If $V(C)$ induces
a non-bipartite subgraph of $H$, then we are done -- take $C$ together with an appropriate chord of $C$. Suppose $V(C)$ induces a
bipartite subgraph of $H$. Since $H$ is weakly $3k$-expanding on independent sets, $\chi(H) \geq \iota(H) > 3k \geq e^{15}$. In
particular, $H - V(C)$ is certainly non-bipartite and contains an odd cycle $D$. By 2-connectivity of $H$, there exist two vertex
disjoint paths $Q_1'$ and $Q_2'$ from $V(D)$ to $V(C)$ whose internal vertices are in $H - V(C) - V(D)$. In particular, there exist
$w_1,w_2 \in V(C)$ such that there are both even and odd length paths $Q_1$ and $Q_2$ respectively and $V(Q_1) \cap V(C) =
\{w_1,w_2\} = V(Q_2) \cap V(C)$. Indeed, let $Q_i$ consist of $Q_1' \cup Q_2'$ together with a subpath of $D$ of length congruent to
$|E(Q_1')| + |E(Q_2')| + i$ modulo two, for $i \in \{1,2\}$. Let $P_1$ and $P_2$ denote the two $w_1w_2$-subpaths of $C$, and suppose
$P_1$ contains at least $\frac{1}{2}T$ special vertices. If $P_1$ has a chord, then $P_1 \cup Q_1$ or $P_1 \cup Q_2$ is the required
non-bipartite subgraph. If $P_1$ has no chord, then at least $\frac{1}{2}T$ special vertices on $P_1$ each have at least two
neighbors in $V(P_2) \backslash \{w_1,w_2\}$. Pick a special vertex $w \in V(P_1)$ and neighbors $x_1,x_2 \in N(w) \cap V(P_2)$. We
assume that the order of appearance of vertices on $C$ is $w_1,w,w_2,x_2,x_1$ clockwise. For $a,b \in V(C)$, let $C(a,b)$ denote the
internal vertices of the subpath of $C$ from $a$ to $b$ in the clockwise order where $a$ precedes $b$. Consider the paths $R_1 = C -
C(x_1,w_1) - C(w,w_2) + \{w,x_1\}$ and $R_2 = C - C(w_2,x_2)- C(w_1,w) + \{w,x_2\}$. The path $R_1$ is shown by dotted lines and
arrows in the figure below. Then $R_1 \cup R_2 = C$ and $\{w,x_2\}$ is a chord of $R_1$ and $\{w,x_1\}$ is a chord of $R_2$. One of
these paths has length at least $\frac{1}{2}|V(C)| > T$, and together with $Q_1$ or $Q_2$ forms the required non-bipartite subgraph.
 \end{proof}

\medskip

\SetLabels
\R(0.54*0.35)$w_1$\\
\R(0.72*0.69)$w_2$\\
\R(0.62*0.32)$x_1$\\
\R(0.75*0.53)$x_2$\\
\R(0.30*0.79)$w$\\
\R(0.50*0.86)$C$\\
\R(0.83*0.16)$D$\\
\endSetLabels
\begin{center}
\centerline{\AffixLabels{\includegraphics[width=4in]{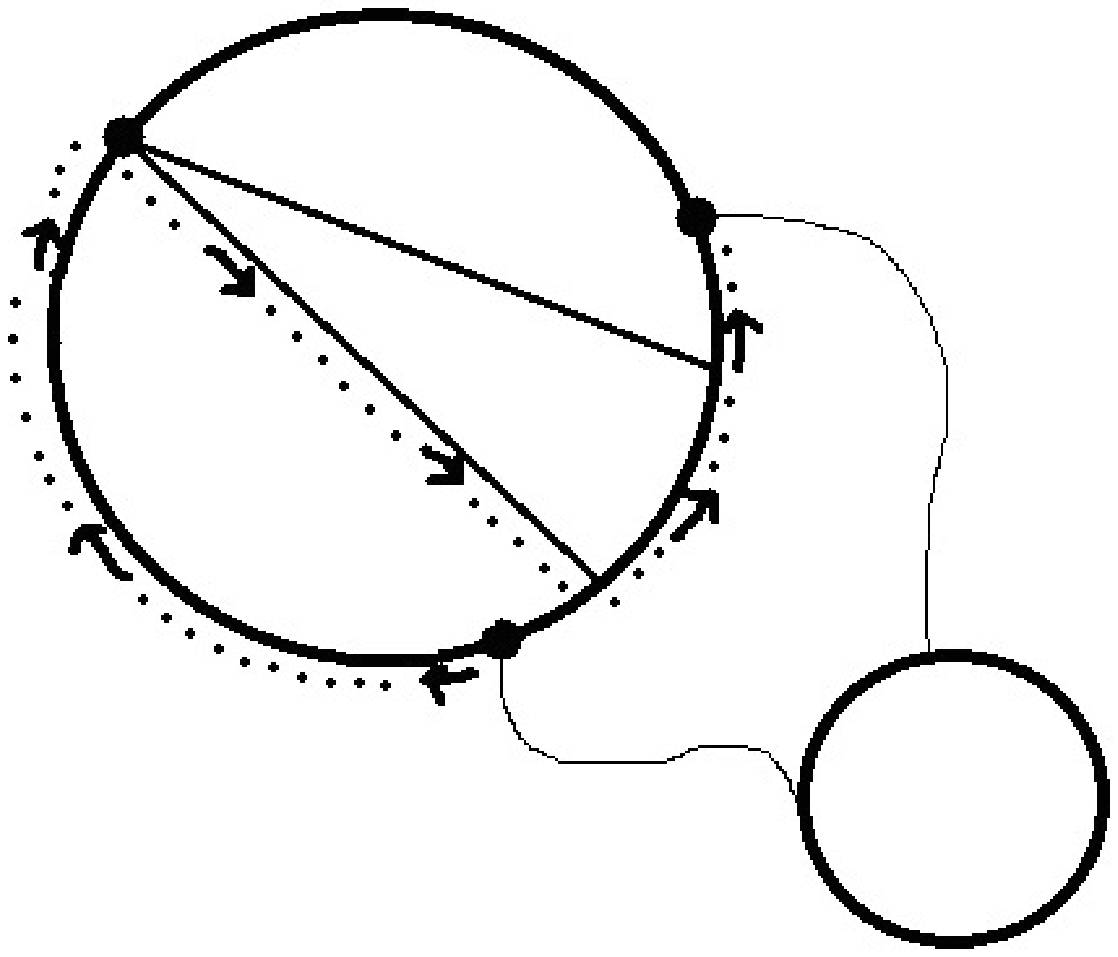}}}
\end{center}

\subsection{Consecutive cycle lengths}

The main ingredient for finding cycles of many consecutive lengths in graphs is the following
proposition~\cite{V}:

\begin{proposition}\label{paths}
Let $F$ be a non-bipartite graph comprising a cycle with a chord, and let $(A,B)$ be a non-trivial partition of $V(F)$.
Then for each $\ell \leq |V(F)| - 1$, there exists a path in $F$ of length $\ell$ with one endpoint in $A$ and the other endpoint in $B$.
\end{proposition}

We illustrate the argument which will be used repeatedly in the rest of the paper. Recall that $N_i(v)$ denotes the
set of vertices at distance exactly $i$ from a given vertex $v$ in a graph.

\begin{proposition}\label{consecutive}
Let $G$ be a graph and suppose that $N_i(v)$ contains a non-bipartite graph $F$ comprising a cycle of length $L$ together with a chord.
Then $G$ contains cycles of $L$ consecutive lengths, the shortest of which has length at most $2i + 1$.
\end{proposition}

\begin{proof}
Consider a breadth-first search tree rooted at $v$, and let $T$ be a minimal subtree
whose set of leaves is $V(F)$. Then $T$ branches at its root, and we let $A$ be the set of leaves of $T$ in one branch, and
$B$ the remaining set of leaves of $T$. Then $A \cup B = V(F)$ and $A,B$ partition $V(F)$. Now we use Proposition \ref{paths}: there exists a path in $F$ of length $\ell$ whose first vertex is in $A$ and whose last
vertex is in $B$ for all $\ell \leq |V(F)| - 1$. If $T$ has height $h$, then we obtain cycles of lengths $2h + 1,2h + 2,\dots,2h + L - 1$
in $H$, as required. Furthermore, since $h \leq i$, the shortest cycle has length $2h + 1 \leq 2i + 1$.
\end{proof}

\subsection{Proof of Theorem \ref{long}}

Theorem \ref{long} states that a triangle-free graph $G$ with $\iota(G) \geq k \geq 3$ contains cycles of $\Omega(k^2 \log k)$
consecutive lengths. It is enough to prove the following quantitative version of this statement: let $G$ be a triangle-free graph with $\iota(G) > 18k + 4$ where $k \geq e^{15}$, then $G$ contains cycles
of at least $k^2 \log k$ consecutive lengths. Pass to a subgraph $F$ of $G$ such that $\alpha(F) < |V(F)|/(18k + 4)$. By Lemma \ref{indep-expand}, $F$ has a subgraph $H$ that is $18k +
3$-expanding on
independent sets. In this subgraph, pick any vertex $v$ and consider $N_i := N_i(v)$ for $i \in \mathbb N$. If for some $i \in \mathbb N$, $\iota(N_i) > 3(2k) + 1$, then Proposition \ref{theta} shows that the
subgraph induced by $N_i(v)$ contains a non-bipartite graph consisting of an odd cycle of length at least $\frac{1}{2}(2k)^2 \log 2k > k^2 \log k$ with at least one chord. In this case, the theorem follows from Proposition
\ref{consecutive}. Otherwise, we have $\iota(N_i) \leq 6k + 1$ for every $i \in \mathbb N$. In that case, $\alpha(N_i) \geq |N_i|/(6k + 1)$ and $N_i$ contains and independent set $I$ of size
at least $|N_i|/(6k + 1)$. Since $H$ is
$18k + 3$-expanding on independent sets, we
conclude
$|\partial I| > 3|N_i|$. In particular, since
\[ \partial I \subset N_i \cup N_{i - 1} \cup N_{i + 1},\]
for all $i \in \mathbb N$ we have
\[ |N_{i + 1}| + |N_{i - 1}| >  3|N_{i}| - |N_i| = 2|N_i|.\]
Then $|N_0| = 1$ and, since $H$ is $18k + 3$-expanding on independent sets, every vertex of $H$ has degree more than $18k + 3$, so $|N_1| > 18k + 3>|N_0|$. We easily obtain $|N_{i + 1}| > |N_i|$ by
induction for all $i \in \mathbb N$, which is clearly impossible since
$H$ is a finite graph so some $N_i$ must be empty. This completes the proof. \hfill{ }\vrule height10pt width5pt depth1pt

\subsection{Hereditary Properties}

In this subsection we prove Theorem \ref{general}. The proof of Theorem \ref{general} is almost identical to the proof of Theorem \ref{long} given over the last three sections, so we merely indicate how to generalize each component of that proof.
Throughout this section, $k \geq 1$ and $f : [1,\infty) \rightarrow [1,\infty)$ is an increasing bijection and $P$ is a hereditary property. For a general hereditary property, the following lemma
generalizes Lemma \ref{monotone-triangle}.

\begin{lemma}\label{monotone}
Let $P$ denote any hereditary property with speed at most $f$ and $G \in P$ where $|V(G)| = n$ and let $k \geq 1$. If $\alpha(G) < n/(6k + 1)$ then $G$ has a 2-connected subgraph that is weakly
$2$-expanding on sets of size at most $f^{-1}(k)$ and weakly $3k$-expanding on independent sets.
\end{lemma}

\begin{proof}
We repeat the proof of Lemma \ref{monotone-triangle}. Pass to a 2-connected subgraph $H \subset G$ which is weakly $6k$-expanding on independent sets, using Lemma \ref{indep-expand}: for some $v \in V(H)$ and every independent set $I \subset V(H) \backslash \{v\}$ we have $|\partial_H I| > 6k|I|$. Let $X \subset V(H)$ with $X \neq \{v\}$, and suppose $|\partial_H X| \leq 2|X|$. If $X = \{x\}$, then $x \neq v$ and $|\partial_H\{x\}| > 6k > 2|X|$ so we have
$|X| \geq 2$.  Since $P$ is a hereditary property of speed at most $f$, there is an independent set $I \subset X \backslash \{v\}$ such that
$|I| \geq (|X| - 1)/f(|X| - 1)$, since the subgraph of $H$ induced by $X$ is in $P$. Consequently,
\[ 2|X| \geq |\partial_H X| \geq |\partial_H I| - |X| + |I| > 6k|I| - |X|\]
and so $2k|I| < |X|$. So if $|\partial_H X| \leq 2|X|$, then
\[ 2k(|X| - 1)/f(|X| - 1) \leq 2k|I| < |X|\]
from which we get $f(|X| - 1) > 2k(|X| - 1)/|X| \geq k$ since $|X| \geq 2$. It follows that $|X| > f^{-1}(k)$.
\end{proof}

Using this lemma, we obtain the following straightforward generalization of Proposition \ref{theta} for hereditary properties.

\begin{proposition}\label{theta2}
Let $P$ be a hereditary property with speed at most $f$ and $G \in P_n$ and let $k \geq 1$. If $\alpha(G) < n/(6k + 1)$, then $G$ contains a non-bipartite graph consisting of a cycle of length at least
$\frac{1}{2}f^{-1}(k)$ with at least one chord.
\end{proposition}

\medskip

{\bf Proof of Theorem \ref{general}.} Let $G \in P$ have $\iota(G) > 18k + 4$ where $k \geq 1$.
By Lemma \ref{indep-expand}, $G$ has a subgraph $H$ that is $18k + 3$-expanding on independent sets. In this subgraph, pick any vertex $v$ and consider $N_i = N_i(v)$ for $i \in \mathbb N$. If for some $i
\in \mathbb N$, $\iota(N_i) > 6k + 1$, then Proposition \ref{theta2} shows that the subgraph induced by $N_i(v)$ contains a non-bipartite graph $F$ consisting of a cycle of length at least
$\frac{1}{2}f^{-1}(k)$
with at least one chord. In this case, the theorem follows from Proposition \ref{consecutive}. Otherwise, we have $\iota(N_i) \leq 6k + 1$ for every $i \in \mathbb N$. In that case, $\alpha(N_i) \geq
|N_i|/(6k + 1)$. Since $H$ is $18k + 3$-expanding on independent sets, we conclude $|N_{i-1} \cup N_i \cup N_{i+1}| > 3|N_i|$, which leads to the contradiction $|N_i| > |N_{i-1}|$ for all $i \in \mathbb
N$, as in Theorem \ref{long}. \hfill{ }\vrule height10pt width5pt depth1pt

\bigskip

\subsection{The property of $H$-free graphs}

To prove Theorem \ref{kskt}, we compute an upper bound for the speed of $P_{s + 1}$, the family of $K_{s + 1}$-free graphs.
Note that if a graph contains $K_{s + 1}$, then it contains a copy of every graph $H$ on $s + 1$ vertices,
and in this case $P_H \subset P_{s + 1}$. We bound the speed using a well-known upper bound on the Ramsey number $r(K_{s + 1},K_{t + 1})$:

\begin{proposition}\label{ramseykskt}
If a graph $G$ has at least ${s + t \choose s}$ vertices, then it contains a clique of order $s + 1$ or an independent
set of order $t + 1$. In particular, $P_{s + 1}$ has speed at most $f_s$
where $f_s(x)= \min\{x,sx^{1 - 1/s}\}$ for $x \in [1,\infty)$.
\end{proposition}

\begin{proof}
The first statement is a well-known bound on Ramsey numbers $r(K_{s+1},K_{t+1})$. Note that ${s + t \choose s} \leq (st)^s$
for all $s,t \geq 1$. It follows that for any $m$-vertex $K_{s + 1}$-free graph $H$, there is an independent
set of size at least $t + 1$ in $H$ whenever $m \geq (st)^s$. By definition, an $n$-vertex graph $K_{s + 1}$-free graph $G$ has
\[ \iota(G) = \sup_{X \subset V(G)} \frac{|X|}{\alpha(X)} \leq \sup_{1 \leq (st)^s \leq n} \frac{(st)^s}{t + 1} < sn^{1 - 1/s}.\]
We also have $\iota(G) \leq n$. This completes the proof.
\end{proof}

For each fixed $s$, the function $f_s$ in the last proposition is a continuous increasing function on $[1,\infty)$ with
$f_s(1) = 1$, and therefore $f_s : [1,\infty) \rightarrow [1,\infty)$ is an increasing bijection. Applying
Theorem \ref{general}, if $G \in P_{s + 1}$ and $\iota(G) > 18k + 4$, then
$G$ contains cycles of at least $\frac{1}{2}f^{-1}(k)$ consecutive lengths. Since
\[ f^{-1}(k) = \max\{k,(k/s)^{s/(s - 1)}\} \geq (k/s)^{s/(s - 1)}\]
we conclude that a graph $G \in P_{s + 1}$ with $\iota(G) > 18k + 4$ has cycles of at least $\frac{1}{2}(k/s)^{s/(s - 1)}$
consecutive lengths. This completes the proof of Theorem \ref{kskt}. $\Box$

\bigskip

{\bf Remark.} By using better bounds on $r(K_{s + 1},K_{t + 1})$, Theorem \ref{kskt} can be improved by logarithmic factors of $k$, as
we achieved for triangles. However computing bounds on the speed of $P_{s + 1}$ is then more cumbersome, so for simplicity
we avoid these calculations.

\section{Proof of Theorem \ref{lengths}}

We are given a sequence $\sigma_r$ satisfying $\log \sigma_r \leq \sigma_{r - 1}$, and we have to prove that
any graph $G$ with $\iota(G) > \sigma_1 \exp(8\log^* n)$ contains a cycle of length $\sigma_r$ for some $r$. We begin with some notation.
Let $\tau$ denote any infinite increasing subsequence of $\sigma$ with $\tau_1 = \sigma_1$ and let $P_r$ denote the property of all graphs containing no cycle of any length $\sigma_j \leq \tau_r$. Define
\[ \triangle_r := \max \{\sigma_j - \sigma_{j - 1} : \sigma_j \leq \tau_r\}.\]
We define $\tau_0 = 1 = \triangle_1$.
To prove Theorem \ref{lengths}, we first prove the following more general theorem:

\bigskip

\begin{theorem}\label{lengths2} Let $\sigma$ be an infinite increasing sequence of positive integers with $\sigma_1 \geq 3$ and let
$\tau$ be an arbitrary subsequence of $\sigma$ with $\tau_1 = \sigma_1$ and $\triangle_r$ defined as above. Then any
$n$-vertex graph $G \in P_r$ has
\[ \iota(G) < 27^r \sigma_1 \exp\Bigl(\sum_{i = 1}^r \frac{2\log \triangle_i}{\tau_{i - 1}} + \frac{2\log n}{\tau_r}\Bigr).\]
\end{theorem}

\medskip

Before proving this theorem, we show how to derive Theorem \ref{lengths} from it:

\bigskip

{\bf Proof of Theorem \ref{lengths}.} In Theorem \ref{lengths}, we are given a sequence $\sigma$ with $\log \sigma_r \leq \sigma_{r - 1}$ for $r \geq 2$.
Let $T(r)$ denote the $r$th element of the sequence
\[ \sigma_1 \quad e^{\sigma_1} \quad e^{e^{\sigma_1}} \quad \cdots\]
Since $\log \sigma_r \leq \sigma_{r - 1}$, there is an element  $\sigma$ between $T(r - 1)$ and $T(r)$ for all $r \geq 2$, which we define to be $\tau_r$.
We now apply Theorem \ref{lengths} with this sequence $\tau = (\tau_r)_{r \in \mathbb N}$ and $\tau_1 = \sigma_1$. Then choose the smallest value of $r$ such that $\tau_r > 2\log n$. Since $T(r)$ is at least a tower of $e$s of height $r$, we note that $r \leq \log^*\!n$. The key fact is that
\[ \sum_{i = 1}^r \frac{2\log \triangle_i}{\tau_{i - 1}} \leq 2r\]
since $\log \triangle_r \leq \tau_{r - 1}$ for all $r \geq 2$. In conclusion, from Theorem \ref{lengths2}, we have
\[ \iota(G) < 27^r \sigma_1 \exp(2r + 1) < \sigma_1 \exp(8\log^*\!n)\]
and this completes the proof. \hfill $\Box$
\bigskip

{\bf Proof of Theorem \ref{lengths2}.} Theorem \ref{lengths2} is a consequence of the following claim:

\medskip

\begin{center}
\parbox{6in}{{\bf Claim 1.} {\it Let $r,m \in \mathbb N$ and $\delta_r = 1/\lceil \tau_r/2\rceil$. Then every $m$-vertex graph $G \in P_r$ has
\[
 \alpha(G) \geq \frac{1}{a_r} m^{1 - \delta_r}.
\]
where $a_r$ is defined by
\[ a_1 = 27\sigma_1  \quad \mbox{ and } \quad a_r = 27a_{r - 1} \triangle_r^{\delta_{r-1}}.\]}
}
\end{center}

\medskip

To see how Theorem \ref{lengths2} follows, take a graph $G \in P_r$ with $n$ vertices; then Claim 1 shows
\[ \iota(G) \leq \sup_{X \subset V(G)} \frac{|X|}{\alpha(X)} \leq \sup_{X \subset V(G)} \frac{|X|}{\frac{1}{a_r} |X|^{1 - \delta_r}}
= a_r n^{\delta_r}\]
since we can apply Claim 1 with $m = |X|$ to the subgraph of $G$ induced by each set $X \subset V(G)$. The linear recurrence $a_r = 27a_{r-1} \triangle_r^{\delta_{r-1}}$ gives
\[ a_r =27^r\sigma_1 \prod_{i=1}^r \triangle_i^{\delta_{i-1}} \leq (27)^r \sigma_1 \exp\Bigl(\sum_{i = 1}^r \frac{2\log \triangle_i}{\tau_{i - 1}}\Bigr)\]
and therefore
\[ \iota(G) \leq a_r n^{\delta_r} \leq 27^r \sigma_1 \exp\Bigl(\sum_{i = 1}^r \frac{2\log \triangle_r}{\tau_{r - 1}} + \frac{2\log n}{\tau_r}\Bigr).\]
We proceed to the proof of Claim 1.

\medskip

{\it Proof of Claim 1.} The claim is true for $P_1$ using early known bounds on cycle-complete Ramsey numbers~\cite{EFRS} (see also~\cite{V}): for $\sigma \in \{2\ell + 1,2\ell + 2\}$,
\[ r(C_L,K_t) < 27\sigma_1 t^{1 + 1/\ell} \]
from which we obtain the required lower bound. Suppose $r > 1$ and that Claim 1 has been proved for every graph in $P_{r-1}$, but that the claim is false for $P_r$.
We first compute an upper bound for the speed of $P_{r - 1}$.

\begin{center}
\parbox{6in}{
{\bf Claim 1.1.} {\it The property $P_{r - 1}$ has speed at most $f$ where
\[ f(x) = \min\{x,a_{r - 1}x^{\delta_{r - 1}}\} \quad \mbox{ for }x \in [1,\infty).\]
}}
\end{center}

{\it Proof of Claim 1.1.} Since Claim 1 holds for $P_{r-1}$, we have that every $p$-vertex graph $F \in P_{r - 1}$ satisfies
\[ \alpha(F) \geq \frac{1}{a_{r-1}} p^{1 - \delta_{r - 1}}\]
This implies
\[ \iota(F) = \sup_{X \subset V(G)} \frac{|X|}{\alpha(X)} \leq a_{r - 1}p^{\delta_{r - 1}}.\]
Therefore $P_{r-1}$ has speed at most $f$ where
\[ f(x) = \min\{x,a_{r - 1}x^{\delta_{r - 1}}\}\]
for $x \geq 1$, as required. \hfill $\Box$

\bigskip

Since by assumption Claim 1 does not hold for $P_r$, for some $m \in \mathbb N$ there exists an $m$-vertex graph $G \in P_r$ such that
\[\alpha(G) < \frac{1}{a_r} m^{1 - \delta_r}.\]
Using Lemma \ref{indep-expand}, pass to an induced subgraph $H$ of $G$ which is $a_r m^{\delta_r}$-expanding on independent sets.
For the rest of the proof of Claim 1, we work in the subgraph $H$ of $G$ to derive a contradiction.

\medskip

\begin{center}
\parbox{6in}{
{\bf Claim 1.2.} {\it For any $v \in V(H)$ and any positive integer $j \leq 1/\delta_r$,
\[
\alpha(N_j(v)) > \frac{3|N_j(v)|}{a_r}.
\]}
}\end{center}

\medskip

{\it Proof of Claim 1.2.} Fix $j \leq 1/\delta_r$ and let $H_j$ be the subgraph of $H$ induced by $N_j(v)$. For convenience, let $h_j = |N_j(v)| = |V(H_j)|$. If $\alpha(H_j) > h_j/9$, then the proof of Claim 1.2 is complete, since $a_r \geq 9 \sigma_1 \geq 27$ for all $r \in \mathbb N$. We may therefore write $\alpha(H_j) = h_j/(6k + 1)$ where $k \geq 1$ is a real number. By Proposition \ref{theta2}, if $f$ is an upper bound for the speed of $P_{r - 1}$, then $H_j$ contains a non-bipartite graph comprising a cycle of length at least $\frac{1}{2}f^{-1}(k)$ with a chord.
It is important to note that $P_r \subset P_{r - 1}$ and therefore $H_j \in P_{r - 1}$.
By Claim 1.1,
\[ f^{-1}(k) = \max\{k,(k/a_{r - 1})^{1/\delta_{r - 1}}\} \geq (k/a_{r - 1})^{1/\delta_{r - 1}}.\]
By Proposition \ref{consecutive}, $H$ contains cycles of at least $\frac{1}{2}f^{-1}(k)$ consecutive lengths,
the shortest of which has length at most $2j + 1 \leq \tau_r$. By definition of $\triangle_r$,
and since $G \in P_r$, we must have $\frac{1}{2}f^{-1}(k) \leq \triangle_r$. 
Recall that $r \geq 2$ and therefore $\delta_{r - 1} \leq 1/2, a_{r-1} \geq 81$.
Rearranging this inequality, and using $\alpha(H_j) = h_j/(6k + 1)$,
we obtain
\[ \alpha(H_j) > \frac{h_j}{9a_{r - 1} \triangle_r^{\delta_{r - 1}}}.\]
By definition of $a_r$, the denominator is $a_r/3$, as required. \hfill $\Box$

\medskip

Since $H$ is $a_r m^{\delta_r}$-expanding on independent sets, we have for any $j \leq 1/\delta_r$ and any maximum independent set $I$ in $H_j$,
\[ |\partial_H N_j(v)| \geq |\partial_H I| > 3m^{\delta_r} \cdot h_j.\]
Since $\partial_H N_j(v) \subset N_j(v) \cup N_{j - 1}(v) \cup N_{j + 1}(v)$, we conclude that
\[ h_{j + 1} + h_{j - 1}  > (3m^{\delta_r} - 1)h_j \geq 2m^{\delta_r}h_j.\]
We also have $h_0 = 1$ and $h_1 > 2 m^{\delta_r}$, since $H$ is $a_r m^{\delta_r}$-expanding on independent sets. Now the recurrence inequality
$h_j + h_{j - 1} > (2c)h_j$ with $h_0 = 1$ and $h_1 > 2c$ has $h_j > c^j$ for all $j$.
With $c = m^{\delta_r}$, we obtain $h_j > m^{j\delta_r}$ for all $j \leq 1/\delta_r$, and in particular we obtain the contradiction
$h_j > m = |V(G)|$ when $j = 1/\delta_r$. This completes the proof of Claim 1, and hence Theorem \ref{lengths2}. \hfill $\Box$

\section{Proof of Theorem \ref{recip}}

We are given an $n$-vertex graph $G$ with $\iota(G) \geq t$ and we have to show
\[ L_{\circ}(G) \geq \frac{1}{2}\log t - 8\log^*\!n.\]
This is clearly true if $t \leq \exp(16\log^*\!n)$ so we assume $t > \exp(16\log^*\!n)$.
We let $s = t/\exp(8\log^*\!n)$, and
consider the disjoint intervals of odd numbers:
\[ S_i = [s^i, s^{i + 1}) \cap (2\mathbb N + 1)\]
for $i \geq 0$. If for some $i$ we have $S_i \subset C(G)$, then
\[ L_{\circ}(G) \geq \frac{1}{2}\log s \geq \frac{1}{2}\log t - 8\log^*\!n \]
and the proof is complete. Otherwise, for each $i$ we pick $\sigma_i \in S_i \backslash C(G)$.
Then we have defined a sequence $\sigma$ with $\sigma_i \leq s \sigma_{i+1}$. Let $\tau$ be a subsequence of $\sigma$
such that $\log \tau_r \leq \tau_{r - 1}$ and such that $\tau_r$ is contained in the interval $S_i$ for which $\sigma_1 T(r) \geq s^{i+1} - 1$ but $\sigma_1 T(r) < s^{i + 2} - 1$.
Here $T(r)$ is a tower of $e$s of height $r$. Note that $\tau_r$ is well defined, since intervals $S_i$ cover all numbers.
We also let $\tau_1 = \sigma_1$. Applying Theorem \ref{lengths2}
with this sequence $\tau$ and choosing $r$ such that $\tau_r > 2\log n$, we have $r \leq \log^*\!n$ and
\[ \iota(G) < 27^r \sigma_1 \exp(2r + 1) < \sigma_1 \exp(8\log^*\!n) < s\exp(8\log^*\!n) < t.\]
This contradiction completes the proof. \hfill $\Box$

\section{Constructions}

In this section, we give constructions which show that the conclusion of Theorem \ref{lengths} does not hold if $\iota(G)$ is replaced with the chromatic number $\chi(G)$.

\bigskip

{\bf Construction.} The existence of triangle-free graphs with arbitrarily large chromatic number was explicitly established by the so-called Mycielski graphs. A survey is given
in~\cite{SS}. These constructions were generalized (see page 213 in~\cite{SS}) to give, for each $k,r \in \mathbb N$ where $r$ is odd, a graph $G_{k,r}$
with chromatic number $k$ and no odd cycle of length at most $r$ and with $|V(G_{k,r})| = 2^{3 - k}(r + 2)^{k - 2}$. We observe that
$|V(G_{k,r})| \leq r^{k} - 1$ for every $r,k \geq 3$, so we let $G'_{k,r}$ consist of $G_{k,r}$ together with enough isolated
vertices so that $|V(G'_{k,r})| = r^k - 1 := n$. Note that $n$ is even. Provided that $\log(n + 1) \leq r$, the increasing sequence $\sigma$ of odd integers in $\{3,5,\dots,r\} \cup \{n + 1,n + 3,\dots\}$
satisfies the
requirements of Theorem \ref{lengths} and yet $G_{k,r}$ has no cycle of length in $\sigma$. For instance, if we take $k = \lfloor r/\log r\rfloor$ and $k \geq 3$, then $n + 1 = r^k \leq e^r$ so $\log(n +
1) \leq r$, and
\[ \chi(G_{k,r}) = k \geq \Big\lfloor \frac{\log (n + 1)}{\log\log(n + 1)} \Big\rfloor \gg \exp(8\log^*\!n) \,.\]
It follows that while $\chi(G_{k,r})$ is substantially larger than the bound $\exp(8\log^*\!n)$ for the independence ratio
in Theorem \ref{lengths}, $G_{k,r}$ has no cycles of length in the sequence $\sigma$.

\section{Concluding remarks}

$\bullet$ We remark that the chromatic number can be arbitrarily large relative to the independence ratio of a graph.
Consider the {\em Kneser graph} $K_{n:r}$
whose vertex set is all subsets of $\{1,2,\dots,n\}$ of size $r$, and whose edges consisting of
pairs of disjoint subsets of $\{1,2,\dots,n\}$ of size $r$. By Lov\'{a}sz's
Theorem~\cite{L}, $\chi(K_{n:r}) = n - 2r + 2$. By the Erd\H{o}s-Ko-Rado Theorem~\cite{EKR},
$\alpha(K_{n:r}) = {n - 1 \choose r - 1}$ and therefore $\iota(K_{n:r}) \geq n/r$. If $n = sr$ where $s > 2$ then
for any set $X \subset K_{n:r}$, we claim $\alpha(X) \leq |X|/s$. To see this, given any collection of $|X|$ subsets
of $\{1,2,\dots,sr\}$ of size $r$, there exists an $i \in \{1,2,\dots,sr\}$ such that
$i$ is contained in at least $r|X|/sr = |X|/s$ of the sets of size $r$. The sets containing $i$ therefore
form an independent set of size at least $|X|/s$ in $K_{n:r}$. Consequently $\iota(K_{sr:r}) \leq s$ and we
conclude $\iota(K_{sr:r}) = s$. Writing $|V(K_{sr:r})| = N$, we obtain $N = {sr \choose r} < (es)^r$ and therefore
\[ \chi(K_{n:r}) > \frac{\iota(G) - 2}{\log \iota(G) + 1} \log N.\]
A good example is the Kneser graph $K_{3r,r}$ with $N = {3r \choose r}$ vertices,
which has independence ratio three and chromatic number $\Theta(\log N)$.

\bigskip

$\bullet$ The Erd\H{o}s-Hajnal conjecture~\cite{E,EH} remains open. A partial step would be to show that if $G$ is a graph with infinite independence
ratio, then $L_{\circ}(G)$ is infinite, whereas in this paper we showed $L_{\circ}(G) > \frac{1}{2}\log \iota(G) - 8\log^*\!n$ when $G$ is an $n$-vertex
graph. Perhaps it is true that $L_{\circ}(G) > \frac{1}{2}\log \chi(G) - O(1)$ for any finite graph $G$, although this is an even stronger
question than the Erd\H{o}s-Hajnal~\cite{EH} conjecture.

\end{document}